\documentclass{amsart}
\usepackage{amssymb}

\newcommand{\Sym}{S}
\newcommand{\FSym}{S_{\w}}
\newcommand{\id}{\mathrm{id}}
\newcommand{\supp}{\mathrm{supt}}
\newcommand{\w}{\omega}
\newcommand{\IN}{\mathbb N}
\newcommand{\IZ}{\mathbb Z}
\newcommand{\A}{\mathcal A}
\newcommand{\e}{\varepsilon}

\newcommand{\M}{\mathfrak M_G}
\newcommand{\Zeta}{\mathfrak Z_G}
\newcommand{\C}{\mathfrak T_G}
\newcommand{\Tau}{\mathcal T}
\newcommand{\Zetap}{\mathfrak Z_G''}
\newcommand{\Zetapp}{\mathfrak Z_G'}
\newcommand{\Homeo}{\mathcal{H}}
\newcommand{\U}{\mathcal U}
\newcommand{\II}{\mathbb I}

\input xy
\xyoption{all}

\newtheorem{theorem}{Theorem}[section]
\newtheorem{conjecture}[theorem]{Conjecture}

\newtheorem{claim}[theorem]{Claim}
\newtheorem{lemma}[theorem]{Lemma}
\newtheorem{proposition}[theorem]{Proposition}
\newtheorem{problem}[theorem]{Problem}
\newtheorem{corollary}[theorem]{Corollary}
\theoremstyle{definition}
\newtheorem{remark}[theorem]{Remark}

\begin{document}

\title{Algebraically determined topologies on permutation groups}
\author{Taras Banakh, Igor Guran, Igor Protasov}
\address{T.Banakh: Ivan Franko National University of Lviv (Ukraine), and Jan Kochanowski University in Kielce (Poland)}
\email{t.o.banakh@gmail.com}
\address{I.Guran: Ivan Franko National University of Lviv (Ukraine)}
\email{igor\_guran@yahoo.com}
\address{I.Protasov: Faculty of Cybernetics, Kyiv University, (Ukraine)}
\email{i.v.protasov@gmail.com}
\dedicatory{Dedicated to Dikran Dikranjan on the occasion of his 60th birthday}

\begin{abstract} In this paper we answer several questions of Dikran Dikranjan about algebraically determined topologies on the groups $\Sym(X)$ (and $\FSym(X)$~) of (finitely supported) bijections of a set $X$. In particular, confirming conjecture of Dikranjan, we prove that the topology $\Tau_p$ of pointwise convergence on each subgroup $G\supset \FSym(X)$ of $\Sym(X)$  is the coarsest  Hausdorff  group topology on $G$ (more generally, the coarsest $T_1$-topology which turns $G$ into a [semi]-topological group), and $\Tau_p$  coincides with the Zariski and Markov topologies $\Zeta$ and $\M$ on $G$. Answering another question of Dikranjan, we prove that the centralizer topology $\C$ on the symmetric group $G=\Sym(X)$ is discrete if and only if $|X|\le\mathfrak c$. On the other hand, we prove that for a subgroup $G\supset\FSym(X)$ of $\Sym(X)$ the centralizer topology $\C$ coincides with the topologies $\Tau_p=\M=\Zeta$ if and only of $G=\FSym(X)$. We also prove that the group $\FSym(X)$ is $\sigma$-discrete in each Hausdorff shift-invariant topology.
\end{abstract}
\maketitle

\section{Introduction}

In this paper we answer several problems of Dikran Dikranjan concerning algebraically determined topologies on the group $\Sym(X)$ of permutations of a set $X$ and its normal subgroup $\FSym(X)$ consisting of all permutations $f:X\to X$ with finite support $\supp(f)=\{x\in X:f(x)\ne x\}$.

The symmetric group $\Sym(X)$ carries a natural group topology, namely the topology of pointwise convergence $\Tau_p$, inherited from the Tychonoff power $X^X$ of the set $X$ endowed with the discrete topology (a topology $\Tau$ on a group $G$ is called a {\em group topology} if $(G,\Tau)$ is a topological group).
Answering a question of Ulam \cite[p.178]{Scot} (cf. \cite{Ulam}), Gaughan \cite{Gau} proved that for each set $X$ the topology $\Tau_p$ is the coarsest Hausdorff group topology on the symmetric group $G=\Sym(X)$ (cf. \cite[1.7.9]{DPS} and \cite[5.2.2]{Dik}). On the other hand, Dierolf and Schwanengel \cite{DiS} proved that for each set $X$ and each subgroup $G\supset\FSym(X)$ of $\Sym(X)$ the topology $\Tau_p$ is a minimal Hausdorff group topology on $G$. In \cite{Gabor} Dikranjan asked if the results of Gaughan and Dierolf, Schwanengel can be unified. More precisely, he made the following conjecture:

\begin{conjecture}[Dikranjan]\label{dikconj} Let $X$ be an infinite set, and $G$ a subgroup of $\Sym(X)$ such that  $\FSym(X)\subset G$. Then the topology of pointwise convergence is the coarsest Hausdorff group topology on $G$.
\end{conjecture}

In this paper we shall confirm this conjecture of Dikranjan. We shall do that by showing that the topology $\Tau_p$  coincides with certain algebraically determined topologies on $G$.

\section{Zariski topologies on permutation groups}

In this section we shall consider three algebraically determined topologies on a group $G$:
\begin{itemize}
\item the {\em Zariski topology} $\Zeta$ generated by the sub-base consisting of the sets
$\{x\in G:x^{\e_1}g_1x^{\e_2}g_2\cdots x^{\e_n}g_n\ne 1_G\}$, where $n\in\IN$, $g_1,\dots,g_n\in G$, and $\e_1,\dots,\e_n\in\{-1,1\}$;
\item a {\em restricted Zariski topology} $\Zetap$, generated by the sub-base consisting of the sets $\{x\in G:xbx^{-1}\ne aba^{-1}\}$ where $a,b\in G$ and $b^2=1_G$;
\item a {\em restricted Zariski topology} $\Zetapp$, generated by the sub-base consisting of the sets
$\{x\in G:xbx^{-1}\ne aba^{-1}\}$ and $\{x\in G:(xcx^{-1})b(xcx^{-1})^{-1}\ne b\}$ where $a,b,c\in G$ and $b^2=c^2=1_G$.
\end{itemize}
Here by $1_G$ we denote the identity element of the group $G$.

It is clear that $\Zetap\subset\Zetapp\subset\Zeta$ and these three topologies on $G$ are {\em shift-invariant} in the sense that for any points $a,b\in G$ the two-sided shift $s_{a,b}:G\to G$, $s_{a,b}:x\mapsto axb$, is continuous. It is also clear that the Zariski topology $\Zeta$ is weaker than an arbitrary Hausdorff group topology on $G$. In particular, $\Zeta\subset\Tau_p$.

The Zariski topology $\Zeta$ is well known in the theory of (topological) groups (see \cite{Br}, \cite{Dik}, \cite{DS}, \cite{DS2}, \cite{DS3}, \cite{DT}, \cite{Sipa}, \cite{BP}, \cite{BPS}) and its origin goes back to Markov \cite{Mar}. On the other hand, the restricted Zariski topologies $\Zetap$ and $\Zetapp$ are less studied and seem to be not used before. Observe that for each Abelian group $G$ the topologies $\Zetap$ and $\Zetapp$ are anti-discrete.

\begin{theorem}\label{main2} For each set $X$ of cardinality $|X|\ge 3$ and each subgroup $G\subset \Sym(X)$ with $\FSym(X)\subset G$ we get $$\Zetap\subset\Zetapp=\Zeta=\mathcal T_p.$$If the set $X$ is infinite, then $\Zetap\ne\Zetapp$.
\end{theorem}

The rest of this section is devoted to the proof of Theorem~\ref{main2}. We assume that $X$ is a set of cardinality $|X|\ge 3$ and $G$ is a subgroup of the symmetric group $\Sym(X)$ such that $\FSym(X)\subset G$. Elements of the set $X$ will be denoted by letters $x,y,a,b,c$, while elements of the symmetric group $\Sym(X)$ by $g,f,h,t,s,u,v,w$. For two points $x,y\in X$ by $t_{x,y}\in\FSym(X)$ we shall denote the transposition which exchanges $x$ and $y$ but  does not move other points of $X$. It is clear that $t_{x,y}$ is a unique permutation $t\in \Sym(X)$ with $\supp(t)=\{x,y\}$.

\begin{lemma}\label{notcom} For any permutation $f\in\Sym(X)$ and any points $x,y\in X$ with $x\ne f(x)$ and $y\notin\{x,f(x)\}$, the transposition $t_{x,y}$ does not commute with $f$, that is $t_{x,y}\circ f\ne f\circ t_{x,y}$.
\end{lemma}

\begin{proof} This follows from $t_{x,y}\circ f(x)=f(x)\ne f(y)=f\circ t_{x,y}(x)$.\end{proof}

Given a subset $A\subset X$, consider the subgroups
$$G(A)=\{g\in G:\supp(g)\subset A\}\mbox{ \ and \ }G_A=\{g\in G:\supp(g)\cap A=\emptyset\}=\{g\in G:g|_A=\id|_A\}$$of $G$. Observe that $G_A=G(X{\setminus}A)$, $G(A)\cap G_A=\{1_G\}$ and any two permutations $f\in G(A)$ and $g\in G_A$ commute because they have disjoint supports: $\supp(f)\cap\supp(g)\subset A\cap(X\setminus A)=\emptyset$.
\smallskip

The definitions of the topologies $\Zeta$ and $\mathcal T_p$ guarantee that those are $T_1$-topologies. The same is true for the topologies $\Zetap$ and $\Zetapp$.

\begin{lemma}\label{l1} The topologies $\Zetap\subset\Zetapp$ satisfy the separation axiom $T_1$.
\end{lemma}

\begin{proof} Given two distinct permutations $f,g\in G$, consider the permutation $h=f^{-1}\circ g\ne 1_G$ and find a point $x\in X$ such that $h(x)\ne x$. Since $|X|\ge 3$, we can choose a point $y\in X\setminus\{x,h(x)\}$ and consider the transposition $t=t_{x,y}$. Then $t^2=1_G$ and $t\circ h\ne h\circ t$ by Lemma~\ref{notcom}.
Now we see that $U=\{u\in G:utu^{-1}\ne t\}$ is a $\Zetap$-open set which contains $h$ but not $1_G$. Then its shift $f\circ U=\{v\in G:vtv^{-1}=ftf^{-1}\}\in\Zetap$ contains $g=f\circ h$ but not $f=f\circ 1_G$.
\end{proof}

If the set $X$ is finite, then the group $G\subset\Sym(X)$ is finite too. In this case the $T_1$-topologies $\Zetap\subset\Zetapp\subset \Zeta\subset\Tau_p$ are discrete and hence coincide.
So, we assume that the set $X$ is infinite.

\begin{lemma}\label{l2a} For each 3-element subset $A\subset X$ the subgroup $G_A$ is $\Zetap$-closed  and so also $\Zetapp$-closed in $G$. \end{lemma}

\begin{proof} Take any permutation $f\in G\setminus G_A$ and find a point $a\in A$ with $f(a)\ne a$.  Since $|A|=3$, we can choose a point $b\in A\setminus\{a,f(a)\}$ and consider the transposition $t_{a,b}$. By Lemma~\ref{notcom}, $t_{a,b}\circ f\ne f\circ t_{a,b}$. Since $\supp(t_{a,b})=\{a,b\}\subset A$, the transposition $t_{a,b}$ commutes with all permutations $g\in G_A$, which implies that $$O_f=\{g\in G:g\circ t_{a,b}\ne t_{a,b}\circ g\}=\{g\in G:gt_{a,b}g^{-1}\ne t_{a,b}\}$$ is a $\Zetap$-open neighborhood of $f$ that does not intersect the subgroup $G_A$ and witnesses that this subgroup is $\Zetap$-closed.
\end{proof}

\begin{lemma}\label{l2} For each 3-element subset $A\subset X$ the subgroup $G_A$ is $\Zetapp$-open.
\end{lemma}

\begin{proof} Assume that for some 3-element subset $A'\subset X$ the subgroup $G_{A'}$ is not $\Zetapp$-open. Since the topology $\Zetapp$ is shift-invariant, the subgroup $G_{A'}$ has empty interior, and being closed by Lemma~\ref{l2a}, is nowhere dense in $(G,\Zetapp)$.

\begin{claim} For each 3-element subset $A\subset X$ the subgroup $G_A$ is closed and nowhere dense in $(G,\Zetapp)$.
\end{claim}

\begin{proof} Choose any permutation $f\in \FSym(X)\subset G$ with $f(A)=A'$ and observe that $G_A=f^{-1}\circ G_{A'}\circ f$ is closed and nowhere dense in $(G,\Zetapp)$ being a (two-sided) shift of the closed nowhere dense subgroup $G_{A'}$.
\end{proof}

\begin{claim}\label{cl2} For each 3-element subset $A\subset X$ and each finite subset $B\subset X$ the subset $G(A,B)=\{f\in G:f(A)\subset B\}$ is nowhere dense in $(G,\Zetapp)$.
\end{claim}

\begin{proof} Since the set of functions from $A$ to $B$ is finite, there is a finite subset $F\subset G(A,B)$ such that for each permutation $g\in G(A,B)$ there is a permutation $f\in F$ with $g|A=f|A$.
Then $f^{-1}\circ g|_A\in G_A$, which implies that
the set
$$G(A,B)=\bigcup_{f\in F}f\circ G_A$$is closed and nowhere dense in $(G,\Zetapp)$, being a finite union of shifts of the closed nowhere dense subgroup $G_A$.
\end{proof}

Now we can finish the proof of Lemma~\ref{l2}. Since $X$ is infinite, we can find two disjoint 3-element subsets $A,B\subset X$.
Claim~\ref{cl2} guarantees that the set
$G(A,A\cup B)\cup G(B,A\cup B)$ is nowhere dense in $(G,\Zetapp)$.
For any points $a\in A$, $b\in B$ consider the transposition $t_{a,b}\in\FSym(X)\subset G$ and put  $T=\{t_{a,b}:a\in A,\;b\in B\}$. For any transpositions $t,s\in T$ with $t\circ s\ne s\circ t$, the set $$U_{t,s}=\{u\in G: (usu^{-1})t(usu^{-1})^{-1}\ne t\}$$is a $\Zetapp$-open neighborhood of $1_G$ by the definition of the topology $\Zetapp$. Since $T$ is finite, the intersection
$$U=\bigcap\{U_{t,s}:t,s\in T,\;ts\ne st\}$$ is a  $\Zetapp$-open neighborhood of $1_G$.
Choose a permutation $u\in U$, which does not belong to the nowhere dense subset $G(A,A\cup B)\cup G(B,A\cup B)$. Then $u(a),u(b)\notin A\cup B$ for some points $a\in A$ and $b\in B$. Choose any point $c\in B\setminus\{b\}$ and consider two non-commuting transpositions $t=t_{a,c}$ and $s=t_{a,b}$.

It follows from $u\in U\subset U_{t,s}$ that the transposition $v=usu^{-1}$ does not commute with the transposition $t$. On the other hand, the support $\supp(v)=\supp(usu^{-1})=u(\supp(s))=u(\{a,b\})=\{u(a),u(b)\}$ does not intersect the set $A\cup B\supset\{a,c\}=\supp(t)$, which implies that $t v= v t$.
This contradiction completes the proof of Lemma~\ref{l2}.
\end{proof}

Now we can prove the first part of Theorem~\ref{main2}.

\begin{lemma}\label{l:eq} $\Zetap\subset \Zetapp=\Zeta=\Tau_p$.
\end{lemma}

\begin{proof} Since $\Zetap\subset\Zetapp\subset \Zeta\subset\Tau_p$, the equality $\Zetapp=\Zeta=\Tau_p$ will follow as soon as we check that $\Tau_p\subset\Zetapp$. Since the topology $\mathcal T_p$ of pointwise convergence is generated by the sub-base consisting of the sets $G(x,y)=\{g\in G:g(x)=y\}$, $x,y\in X$, it suffices to show that any such set $G(x,y)$ is $\Zetapp$-open. Choose any permutation $f\in \FSym(X)\subset G$ with $f(y)=x$ and observe that the shift $f\circ G(x,y)=G(x,x)=G_{\{x\}}$ is a subgroup of $G$, which contains the subgroup $G_A$ for any 3-element subset $A\subset X$ with $x\in A$. By Lemma~\ref{l2}, the subgroup $G_A$ is $\Zetapp$-open, and so are the subgroup $G(x,x)\supset G_A$ and its shift $G(x,y)=f^{-1}\circ G(x,x)$.
\end{proof}

Finally, we prove the second part of Theorem~\ref{main2}.

\begin{lemma}\label{ZneZ} If the set $X$ is infinite, then $\Zetap\ne\Tau_p.$
\end{lemma}

\begin{proof} Assume for a contradiction, that $\Zetap=\Tau_p$. Then for any point $a\in X$,   the $\Tau_p$-open neighborhood $G(a,a)=\{g\in G:g(a)=a\}$ of $1_G$ is $\Zetap$-open. Consequently, $$1_G\in \bigcap_{i=1}^n\{v\in G:vf_iv^{-1}\ne g_if_ig_i^{-1}\}\subset G(a,a)$$for some permutations $f_1,g_1,\dots,f_n,g_n\in G$ such that $f_i^2=1_G\ne f_i$ for all $i\le n$. We can assume that the permutations $f_1,\dots,f_n$ are ordered so that there is a number $k\in\w$ such that  a permutation $f_i$, $1\le i\le n$, has finite support $\supp(f_i)$ if and only if $i>k$. Consider the finite set $$F=\{a\}\cup\bigcup_{i=k+1}^{n}
\big(\supp(f_i)\cup g_i(\supp(f_i))\big)$$and choose any injective function $u_0:F\to X\setminus F$.  For every $i\in\{1,\dots,k\}$ by induction choose two points $x_i\in\supp(f_i)$ and $y_i\in X$ such that for the finite sets
$$X_{<i}=\{x_j:1\le j<i\} \cup\{f_j(x_j):1\le j<i\}\mbox{ \ and \ }Y_{<i}=\{g_j(x_j):1\le j<i\}\cup\{y_j:1\le j<i\}$$ the following conditions hold
\begin{enumerate}
\item $x_i\notin F\cup X_{<i}\cup f_i^{-1}(F\cup X_{<i})\cup g_i^{-1}(u_0(F)\cup Y_{<i})$;
\item $y_i\notin \{g_i(x_i),g_i(f_i(x_i))\}\cup u_0(F)\cup Y_{<i}$.
\end{enumerate}
The choice of the points $x_i,y_i$, $1\le i\le k$, allows us to find a permutation $u\in\FSym(X)$ such that $$u|F=u_0,\;\;u(x_i)=g_i(x_i)\mbox{ \ and \ }u(f_i(x_i))=y_i$$ for every $i\in\{1,\dots,k\}$.

We claim that $u\circ f_i\circ u^{-1}\ne g_i\circ f_i\circ g_i$ for every $i\in\{1,\dots,n\}$. If $i>k$, then the permutation $uf_iu^{-1}\ne 1_G$ has non-empty support $\supp(u\circ f_i\circ u^{-1})=u(\supp(f_i))\subset u_0(F)$  disjoint with $F\supset g_i(\supp(f_i))=\supp(g_if_ig_i^{-1})$, which implies that $u f_i u^{-1}\ne g_i f_i g_i^{-1}$. If $i\le k$, then
for the point $g_i(x_i)=u(x_i)$ we get
$$u\circ f_i\circ u^{-1}(g_i(x_i))=u\circ f_i(x_i)=y_i\ne g_i\circ f_i(x_i)=g_i\circ f_i\circ g_i^{-1}(g_i(x_i)).$$
Now we see that
$$u\in\bigcap_{i=1}^n\{v\in G:vf_iv^{-1}\ne g_if_ig_i^{-1}\}\subset G(a,a)$$which is not possible as $u(a)=u_0(a)\subset u_0(F)\subset X\setminus F\subset X\setminus\{a\}$.
\end{proof}

\section{The coincidence of Markov and Zariski topologies on permutation groups}

In this section we study the interplay between Zariski and Markov topologies on permutation groups and their subgroups.
By definition, the {\em Markov topology} $\M$ on a group $G$ is the intersection of all Hausdorff group topologies on $G$. The Markov topology $\M$ was explicitly introduced in \cite{DS2}, \cite{DS3} and studied in \cite{DS}, \cite{DT}.

It is clear that $\Zeta\subset \M$ for any group $G$.
By a classical result of Markov \cite{Mar}, for each countable group $G$, the topologies $\Zeta$ and $\M$ coincide. The equality $\Zeta=\M$ also holds for each Abelian group $G$; see  \cite{DS}, \cite{DS3}. Dikranjan and Shakhmatov in Question 38(933) of \cite{DS2} asked if the topologies $\Zeta$ and $\M$ coincide on each symmetric group $G=\Sym(X)$. The following corollary of Theorem~\ref{main2} answers this problem affirmatively.

\begin{corollary}\label{ZarMar} For any set $X$ and any subgroup $G\supset \FSym(X)$ of the symmetric group $\Sym(X)$ we get $\Zeta=\M$.
\end{corollary}

\begin{proof} If $X$ is finite, then the topologies $\Zeta$ and $\M$ coincide, being $T_1$-topologies on the finite group $G$. If $X$ is infinite, then the trivial inclusion $\Zeta\subset\M\subset\Tau_p$ and the non-trivial equality $\Zeta=\Tau_p$ established in Theorem~\ref{main2} imply that $\Zeta=\M$.
\end{proof}

Corollary~\ref{ZarMar} implies the following corollary that solves Question 40(i) of \cite{DS2} and Question 8.4(i) of \cite{DT}.

\begin{corollary} For each uncountable set $X$ the symmetric group $G=\Sym(X)$ has $\Zeta=\M$ but contains a subgroup $H\subset G$ with $\mathfrak Z_H\ne\mathfrak M_H$.
\end{corollary}

\begin{proof} By Corollary~\ref{ZarMar}, the  group $G=\Sym(X)$ has $\Zeta=\M$. Since $X$ is uncountable, the symmetric group $G=\Sym(X)$ contains an isomorphic copy of each group of cardinality $\le\w_1$, according to the classical Cayley's Theorem \cite[1.6.8]{Rob}. In particular, $G$ contains an isomorphic copy of the group $H$ of cardinality $|H|=\w_1$ with $\mathfrak Z_H\ne\mathfrak M_H$, constructed by Hesse \cite{Hesse} (see also \cite[Theorem 3.1]{DT}).
\end{proof}

\section{The minimality of the topology of pointwise convergence on permutation groups}

In this section we shall confirm Dikranjan's Conjecture~\ref{dikconj} and shall prove that for each subgroup $G\supset\FSym(X)$ of the symmetric group $\Sym(X)$ the topology of pointwise convergence is the coarsest Hausdorff group topology on $G$. In fact, we shall do that in a more general context of [semi]-topological groups.

We shall say that a group $G$ endowed with a topology is
\begin{itemize}
\item a {\em topological group} if the function $q:G\times G\to G$, $q:(x,y)\mapsto xy^{-1}$, is continuous;
\item a {\em paratopological group} if the function $s:G\times G\to G$, $s:(x,y)\mapsto xy$, is continuous;
\item a {\em quasi-topological group} if the function $q:G\times G\to G$, $q:(x,y)\mapsto xy^{-1}$, is separately continuous;
\item a {\em semi-topological group} if the function $s:G\times G\to G$, $s:(x,y)\mapsto xy$, is separately continuous;
\item a {\em \textup{[}quasi\textup{]}-topological group} if the functions $q:G\times G\to G$, $q:(x,y)\mapsto xy^{-1}$, and $[\cdot\cdot]:G\times G\to G$, $[\cdot\cdot]:(x,y)\mapsto xyx^{-1}y^{-1}$, are separately continuous;
\item a {\em \textup{[}semi\textup{]}-topological group} if the functions $s:G\times G\to G$, $s:(x,y)\mapsto xy$, and $[\cdot\cdot]:G\times G\to G$, $[\cdot\cdot]:(x,y)\mapsto xyx^{-1}y^{-1}$, are separately continuous.
\end{itemize}
These notions relate as follows:
$$
\xymatrix{
\mbox{topological group}\ar@{=>}[r]\ar@{=>}[d]&\mbox{[quasi]-topological group}\ar@{=>}[r]\ar@{=>}[d]&\mbox{quasi-topological group}\ar@{=>}[d]\\
\mbox{paratopological group}\ar@{=>}[r]&\mbox{[semi]-topological group}\ar@{=>}[r]&\mbox{semi-topological group}}
$$
Observe that a semi-topological (resp. quasi-topological) group $G$ is [semi]-topological (resp. [quasi]-topological) if and only if for every $a\in G$ the function $$\gamma_a:G\to G,\;\;\gamma_a:x\mapsto xax^{-1}$$ is continuous. In the sequel such a function $\gamma_a$ will be called a {\em conjugator}. In \cite{Kap}  a quasi-topological (resp. a [quasi]-topological) group $(G,\tau)$ whose topology $\tau$ satisfies the separation axiom $T_1$ is called a {\em $T_1$-group} (resp. a {\em $C$-group}).

It is easy to see that $(G,\Zetap)$ and  $(G,\Zetapp)$ are quasi-topological groups, while $(G,\Zeta)$ and $(G,\M)$ are [quasi]-topological groups.

\begin{proposition}\label{p1} Let $\Tau$ be a topology on a group $G$.
\begin{enumerate}
\item If $\Tau$ is a $T_1$-topology and $(G,\Tau)$ is a \textup{[}semi\textup{]}-topological group, then $\Zetapp\subset\Tau$.
\item If $\Tau$ is a $T_2$-topology and $(G,\Tau)$ is a semi-topological group, then $\Zetap\subset\Tau$.
\end{enumerate}
\end{proposition}

\begin{proof} 1. Assume that $\Tau$ is a $T_1$-topology and $(G,\Tau)$ is a [semi]-topological group. Fix any elements $a,b,c\in G$ with $b^2=c^2=1_G$ and observe that the conjugator $\gamma_b:G\to G$, $\gamma_b:x\mapsto xbx^{-1}$, is $\Tau$-continuous, being the composition $\gamma_b=s_b\circ [\cdot,b]$ of two $\Tau$-continuous functions $[\cdot,b]:x\mapsto xbx^{-1}b^{-1}$ and $s_b:y\mapsto yb$. Consequently, the set
$$\gamma_b^{-1}(G\setminus\{aba^{-1}\})=\{x\in G:xbx^{-1}\ne aba^{-1}\}\in\Zetap\subset\Zetapp$$is $\Tau$-open.
In particular, the set $U=\{x\in G:xbx^{-1}\ne b\}$ is $\Tau$-open.

Next, consider the $\Tau$-continuous function $\gamma_c:G\to G$, $\gamma_c:x\mapsto xcx^{-1}$, and observe that the set $$\gamma_c^{-1}(U)=\gamma_c^{-1}\big(\{y\in G:yby^{-1}\ne b\}\big)= \{x\in G:(xcx^{-1})b(xcx^{-1})^{-1}\ne b\}\in\Zetapp$$is $\Tau$-open too. Now we see that $\Zetapp\subset\Tau$ because all sub-basic sets of the topology $\Zetapp$ belong to $\Tau$.
\smallskip

2. Assume that $\Tau$ is a $T_2$-topology and $(G,\Tau)$ is a semi-topological group.
 We should prove that for any elements $a,b\in G$ with $b^2=1_G$ the sub-basic set $U=\{x\in G:xbx^{-1}\ne aba^{-1}\}\in\Zetap$ is $\Tau$-open. Fix any point $x\in U$ and observe that $xb\ne cx$ where $c=aba^{-1}$. Since
the topology $\Tau$ is Hausdorff, the distinct points $xb$ and $cx$ of the group $G$ have disjoint $\Tau$-open neighborhoods $O_{xb}$ and $O_{cx}$. The separate continuity of the group operation yields a neighborhood $O_x\in\Tau$ of the point $x$ such that $O_x\cdot b\subset O_{xb}$ and $c\cdot O_x\subset O_{cx}$.
Then $O_x\subset U$, witnessing that the set $U$ is $\Tau$-open.
\end{proof}

The following theorem confirms Dikranjan's Conjecture~\ref{dikconj} and generalizes results of Gaughan \cite{Gau} and Dierolf, Schwanengel \cite{DiS}.

\begin{theorem}\label{main} For each set $X$ and each subgroup $G\supset\FSym(X)$ of the symmetric group $\Sym(X)$,  the topology $\Tau_p$ of pointwise convergence on $G$ is the coarsest $T_1$-topology turning $G$ into a \textup{[}semi\textup{]}-topological group.
\end{theorem}

\begin{proof} Let $\Tau$ be a $T_1$-topology on $G$ turning $G$ into a [semi]-topological group. By Proposition~\ref{p1}(1), $\Zetapp\subset\Tau$ and by Theorem~\ref{main2}, $\Tau_p=\Zetapp\subset\Tau$.
\end{proof}

\begin{remark} Observe that for an infinite set $X$ and any subgroup $G\supset \FSym(X)$ of $\Sym(X)$ the restricted Zariski topology $\Zetap$ is a $T_1$-topology turning $G$ into a quasi-topological group, nonetheless, $\Tau_p\not\subset\Zetap$ by Theorem~\ref{main2}. This example shows that Theorem~\ref{main} cannot be generalized to the class of semi-topological or  quasi-topological groups.
\end{remark}

\section{Centralizer topology on permutation groups}

In this section we study the properties of the centralizer topology $\C$ on permutation groups $G$.
This topology was introduced by Ta\u\i manov in \cite{Tai} (cf. \cite{BD}, \cite{DT}) with the aim of topologizing non-commutative groups.

For a group $G$ its {\em centralizer topology} $\C$ is generated by the sub-base consisting of the sets
$$\{x\in G:xbx^{-1}=aba^{-1}\},$$ where $a,b\in G$. The centralizer topology $\C$ has a neighborhood base at $1_G$ consisting of the centralizers $$c_G(A)=\bigcap_{a\in A}\{x\in G:xa=ax\}$$ of finite subsets $A\subset G$. These centralizers $c_G(A)$ are open-and-closed subgroups of $G$, which implies that $\Zetap\subset\C$ for each group $G$.
By \cite[\S4]{BD}, a group $G$ endowed with its centralizer topology is a topological group. This topological group is Hausdorff if and only if the group $G$ has trivial center $c_G(G)=\{1_G\}$; see \cite[4.1]{BD}. In this case, $\Zetap\subset\Zetapp\subset\Zeta\subset\M\subset \C$. Theorem~\ref{main2} and Lemma~\ref{l1} imply that
$\Zetap\subset\Zetapp=\Zeta=\M=\Tau_p\subset \C$ for any permutation group $G\subset \Sym(X)$ with $\FSym(X)\subset G$ on a set $X$ of cardinality $|X|\ge 3$. Unlike the algebraically determined topologies $\Zetapp=\Zeta=\M$, the centralizer topology $\C$ depends essentially on the position of the group $G$ in the interval between the groups $\FSym(X)$ and $\Sym(X)$. In the extremal case $G=\Sym(X)$ the centralizer topology $\C$ is close to being discrete, as shown by the following theorem, which generalizes Theorem 4.18 of \cite{BD} and answers affirmatively Question 4.19 of \cite{BD} and Question 8.17 of \cite{DT}. In this theorem by $\mathfrak c$ we denote the cardinality of the continuum.

\begin{theorem}\label{tai} For a set $X$ of cardinality $|X|\ge 3$, the centralizer topology $\C$ on the symmetric group $G=\Sym(X)$ is discrete if and only if $|X|\le\mathfrak c$.
\end{theorem}

\begin{proof} Since $|X|\ge 3$, the group $G=\Sym(X)$ has trivial center, which implies that $(G,\C)$ is a Hausdorff topological group; see \cite[\S4]{BD}. If $X$ is finite, then the Hausdorff topological group $(G,\C)$ is finite and hence the centralizer topology $\C$ is discrete.

So, we assume that the set $X$ is infinite. If $|X|>\mathfrak c$, then the centralizer topology $\C$ on the group $G=\Sym(X)$ is not discrete by Theorem 4.18(2) of \cite{BD}. If $|X|=\w$, then the centralizer topology $\C$ is discrete by Theorem 4.18(1) of \cite{BD}.
So, it remains to show that $\C$ is discrete if $\w<|X|\le\mathfrak c$.
By Proposition 4.17(c) of \cite{BD} the discreteness of $\C$ will follows as soon as we construct a finitely generated group $H$ containing continuum many subgroups $H_\alpha\subset H$, $\alpha\in\mathfrak c$, which are {\em self-normalizing} in the sense that for an element $h\in H$, the equality $hH_\alpha h^{-1}=H_\alpha$ holds if and only if $h\in H_\alpha$. Such a group $H$ will be constructed in the following lemma.
\end{proof}

\begin{lemma} There is a finitely generated group $H$ containing continuum many self-normalizing subgroups.
\end{lemma}

\begin{proof}
Consider the free group $F_\IZ$ with countably many generators, identified with integers. Then the shift $\varphi:\IZ\to \IZ$, $\varphi:n\mapsto n+1$, extends to an automorphism $\Phi:F_\IZ\to F_\IZ$ of the free group $F_\IZ$. Let $H=F_\IZ\rtimes_{\Phi}\IZ$ be the semi-direct product of the free group $F_\IZ$
and the additive group $\IZ$ of integers. Elements of the group $H$ are pairs $(v,n)\in F_\IZ\times \IZ$ and the group operation is given by the formula
$$(v,n)\cdot (u,m)=(v\cdot\Phi^n(u),n+m).$$
We shall identify $F_\IZ$ and $\IZ$ with the subgroups $F_\IZ\times\{0\}$ and $\{\mathbf 1\}\times\IZ$, where $\mathbf 1$ stands for the identity element of the free group $F_\IZ$.
Observe that the group $H$ is finitely-generated: it is generated by two elements, $(\mathbf 1,1)$ and $(z,0)$, where $z\in\IZ\subset F_\IZ$ is one of the generators of the free group $F_\IZ$.

Following \cite{LP}, we call subset $A\subset\IZ$ {\em thin} if for each $n\in\IZ\setminus\{0\}$ the intersection $A\cap (A+n)$ is finite. It is easy to check that the family $\A$ of all infinite thin subsets of $\IZ$ has cardinality $|\A|=\mathfrak c$. For each infinite thin set $A\in\A$ denote by $F_A$ the (free) subgroup generated by the set $A\subset \IZ$ in the free group $F_\IZ=F_\IZ\times\{0\}\subset  H$. It remains to prove that the subgroup $F_A$ is self-normalizing in $H$.

Given any element $h=(u,n)\in F_\IZ\rtimes_\Phi\IZ=H$ with $hF_Ah^{-1}=F_A$, we need to prove that $h\in F_A$. First we show that $n=0$. Assume for a contradiction that $n\ne 0$. Find a finite subset $B\subset\IZ$ such that $u\in F_B$ and consider the set $C=B\cup (A+n)$. It follows from our assumption that the intersection $A\cap C=(A\cap B)\cup (A\cap (A+n))$ is finite and hence $A\cap C\ne A$ and $F_{C\cap A}\ne F_A$.

Taking into account that $h^{-1}=(\Phi^{-n}(u^{-1}),-n)$, we see that for each word $w\in F_A\subset F_\IZ\subset H$ $$hwh^{-1}=(u,n)\cdot (w,0)\cdot (\Phi^{-n}(u^{-1}),-n)=(u\Phi^{n}(w),n)\cdot (\Phi^{-n}(u^{-1}),-n)=(u\Phi^n(w)u^{-1},0)\in F_C\cap F_A=F_{A\cap C},$$which implies that
$F_A=hF_Ah^{-1}\subset F_{A\cap C}\subset F_A$ and $F_{A\cap C}=F_A$. This contradiction proves that $n=0$ and hence the element $h=(u,0)$ can be identified with the element $u\in F_\IZ$. Now it is easy to see that the equality $hF_Ah^{-1}=F_A$ implies that $uF_Au^{-1}=F_A$ and hence $h=u\in F_A$.
\end{proof}

By Theorem~\ref{tai}, for the symmetric group $G=\Sym(X)$ of an infinite set $X$ of cardinality $|X|\le\mathfrak c$, the centralizer topology $\C$ is discrete and hence does not coincide with the topology $\Tau_p=\M=\Zeta=\Zetapp$. For the group $G=\FSym(X)$ the situation is totally different.

\begin{theorem}\label{cent} For a set $X$ of cardinality $|X|\ge 3$ and a subgroup $G\supset \FSym(X)$ of the symmetric group $\Sym(X)$, the equality $\Tau_p=\C$ holds if and only if $G=\FSym(X)$.
\end{theorem}

\begin{proof} To prove the ``only if'' part, assume that $G\ne \FSym(X)$ and find a permutation $g\in G$ with infinite support. Assuming that $\Tau_p=\C$, we conclude that the $\C$-open set $c_G(g)=\{f\in G:f\circ g=g\circ f\}$ is $\Tau_p$-open. So, we can find a finite subset $A\subset X$ such that $G_A\subset c_G(g)$. Since the set $\supp(g)\subset X$ is infinite, there are points $x\in\supp(g)\setminus A$ and $y\in X\setminus (A\cup\{x,g(x)\})$. Consider the transposition $t=t_{x,y}$ with $\supp(t)=\{x,y\}\subset X\setminus A$ and observe that $t\circ g\ne g\circ t$ by Lemma~\ref{notcom}. But this contradicts the inclusion $t\in G_A\subset c_G(g)$. So, $\Tau_p\ne\C$.
\smallskip

To prove the ``if'' part, assume that $G=\FSym(X)$. In this case we shall show that $\C=\Tau_p$. Denote by $[X]^{\ge 3}$ the family of all finite subsets $A\subset X$ with $|A|\ge 3$ and observe that the subgroups
$G(X{\setminus}A)$, $A\in[X]^{\ge 3}$, form a neighborhood basis of the topology $\mathcal T_p$ at $1_G$, while the centralizers $c_G\big(G(A)\big)$ of the finite subgroups
$G(A)=\{g\in G:\supp(g)\subset A\}$,  $A\in[X]^{\ge 3}$, form a neighborhood basis of the centralizer topology $\C$ at $1_G$. The following lemma implies that these neighborhoods bases coincide, so $\mathcal T_p=\C$.
\end{proof}

\begin{lemma}\label{l3.3} Let $G=\FSym(X)$ be the group of finitely supported permutations of a set $X$. Then $G(X{\setminus}A)=c_G\big(G(A)\big)$ for each subset $A\subset X$ of cardinality $|A|\ge 3$.
\end{lemma}

\begin{proof} The inclusion $G(X{\setminus}A)\subset c_G\big(G(A)\big)$ trivially follows from the fact that any two permutations with disjoint supports commute.
To prove the reverse inclusion, fix any permutation $f\in c_G\big(G(A)\big)$. Assuming that $f\notin G(X{\setminus}A)$, we can find a point $a\in\supp(f)\cap A$. Since $|A|\ge 3$, there is a point $b\in A\setminus\{a,f(a)\}$. Now consider the transposition $t=t_{a,b}\in\FSym(X)=G$ with $\supp(t)=\{a,b\}$ and observe that $f$ and $t$ do not commute according to Lemma~\ref{notcom}. On the other hand, $f$ should commute with $t$ as $\supp(t)\subset A$ and $f\in c_G\big(G(A)\big)\subset c_G(t)$.
\end{proof}

Since the topologies $\Zeta$ and $\C$ are determined by the algebraic structure of a group $G$, Theorem~\ref{cent} implies the following algebraic fact (for which it would be interesting to find an algebraic proof).

\begin{corollary} For each set $X$, a subgroup $G\supset\FSym(X)$ of the group $\Sym(X)$ is isomorphic to $\FSym(X)$ if and only if $G=\FSym(X)$.
\end{corollary}

%Observe that for a group $G$ with trivial center its centralizer topology $\C$ is discrete if and only if for some finite subset $F\subset G$ its double centralizer $c_G(c_G(F))$ equals $G$.

The permutation groups $\FSym(X)$ belong to the class of groups with finite double centralizers.
We shall say that a group $G$ has {\em finite double centralizers} if for each finite subset $F\subset G$ its double centralizer $c_G(c_G(F))$ is finite. This definition implies that each group with finite double centralizers is locally finite.

\begin{proposition}\label{p3.5} For each set $X$ the permutation group $G=\FSym(X)$ has finite double centralizers.
\end{proposition}

\begin{proof} This proposition is trivial if $X$ is finite. So, we assume that $X$ is infinite. We need to show that for each finite subset $F\subset G$ its double centralizer $c_G(c_G(F))$ in $G$ is finite. Choose any finite subset $A\supset X$ of cardinality $|A|\ge 3$, which contains the (finite) support $\supp(f)$ of each permutation $f\in F$. The subgroup $G(A)=\{f\in G:\supp(f)\subset A\}$ is finite and contains the finite subset $F$. By Lemma~\ref{l3.3}, $G(X\setminus A)=c_G(G(A))\supset c_G(F)$. Since $|X\setminus A|\ge 3$, we can apply Lemma~\ref{l3.3} once more and conclude that $$G(A)=c_G\big(G(X\setminus A)\big)=c_G\big(c_G(G(A))\big)\supset c_G(c_G(F)),$$
which implies that the double centralizer $c_G(c_G(F))$ of $F$ is finite.
\end{proof}

We are interested in groups with finite double centralizers because of the following their property.

\begin{proposition}\label{p3.6} For any infinite group $G$ with finite double centralizers the centralizer topology $\C$ is not discrete.
\end{proposition}

\begin{proof} If the centralizer topology $\C$ is discrete, then the trivial subgroup $\{1_G\}$ of $G$ is $\C$-open, which means that $\{1_G\}=c_G(F)$ for some finite subset $F\subset G$. In this case the double centralizer $c_G(c_G(F))=c_G(1_G)=G$ is infinite, which is a desired contradiction.
\end{proof}

\section{Hausdorff shift-invariant topologies on the groups $\FSym(X)$}

In this section we establish some common properties of Hausdorff shift-invariant topologies on  permutation groups $G=\FSym(X)$. By Proposition~\ref{p1}(2), each Hausdorff shift-invariant topology on $G$ contains the restricted Zariski topology $\Zetap$. Our main result concerning this topology is:

\begin{theorem}\label{t4} Let $X$ be a set of cardinality $|X|\ge 3$ and $G=\FSym(X)$. Then for every $n\in\w$
\begin{enumerate}
\item the subset $S_{\le n}(X)=\{g\in G:|\supp(g)|\le n\}$ is closed in $(G,\Zetap)$;
\item the subspace $S_{=n}(X)=\{g\in G:|\supp(g)|=n\}$ of $(G,\Zetap)$ is discrete.
\end{enumerate}
\end{theorem}

\begin{proof} By Lemma~\ref{l1}, the topology $\Zetap$ on the group $G$ satisfies the separation axiom $T_1$. If the set $X$ is finite, then the topology $\Zetap$ is discrete, being a $T_1$-topology on the finite group $G=\FSym(X)$. So, we assume that $X$ is infinite.
\smallskip

\begin{lemma}\label{l5} For every $x\in X$ the set $U=\{g\in S_{\le n}(X):x\in\supp(g)\}$ is relatively $\Zetap$-open in $S_{\le n}(X)$.
\end{lemma}

\begin{proof} Fix any permutation $g\in U$. It follows that $g(x)\ne x\in \supp(g)$.
Since $X$ is infinite and $|\supp(g)|\le n$, we can choose a subset $A\subset X\setminus \supp(g)$ of cardinality $|A|=n+1$. For each point $a\in A$ consider the transposition $t_{x,a}$ with $\supp(t_{x,a})=\{x,a\}$ and observe that $t_{x,a}\circ g\ne g\circ t_{x,a}$ by Lemma~\ref{notcom}. Then
$$O_g=\bigcap_{a\in A}\{f\in G:f\circ t_{x,a}\circ f^{-1}\ne t_{x,a}\}$$is an $\Zetap$-open neighborhood of $g$. We claim $O_g\cap S_{\le n}(X)\subset U$. This inclusion will follow as soon as for each  permutation $f\in O_g\cap S_{\le n}(X)$ we show that $x\in\supp(f)$. Assume conversely that $x\notin\supp(f)$. Since $|\supp(f)|\le n<|A|$, there is a point $a\in A\setminus\supp(f)$. Then the support $\supp(f)$ is disjoint with the support $\{x,a\}$ of the transposition $t_{x,a}$, which implies that $f$ commutes with $t_{x,a}$. But this contradicts the choice of $f\in O_g$.
\end{proof}

Now we can finish the proof of Theorem~\ref{t4}.
\smallskip

1. To show that the subset $S_{\le n}(X)$ is $\Zetap$-closed, fix any permutation $g\in \FSym(X)\setminus S_{\le n}(X)$. We need to find a neighborhood $O_g\in\Zetap$ of $g$ with $O_g\cap S_{\le n}(X)=\emptyset$. Consider the support $A=\supp(g)$, which is a finite set of cardinality $|A|>n$. The infinity of the set $X$ allows us to chose an injective map $\alpha:A\times \{0,\dots,n\}\to X\setminus A$. Now for each point $a\in A$ and $k\in\{0,1,\dots,n\}$ consider the transposition $t=t_{a,\alpha(a,k)}$ with $\supp(t)=\{a,\alpha(a,k)\}$ and observe that $g\circ t\ne t\circ g$ by Lemma~\ref{notcom}.
So, $$O_g=\bigcap_{i=0}^n\bigcap_{a\in A}\{f\in G:f\circ t_{a,\alpha(a,k)}\circ f^{-1}\ne t_{a,\alpha(a,k)}\}$$ is a $\Zetap$-open neighborhood of $g$. We claim that $O_g\cap S_{\le n}(X)=\emptyset$. Assume for a contradiction that this intersection contains some permutation $f$. Since $|\supp(f)|\le n<|\supp(g)|=|A|$, we can choose a point $a\in A\setminus\supp(f)$ and then choose  a number $k\in\{0,\dots,n\}$ such that $\alpha(a,k)\notin\supp(f)$. Now consider the transposition $t=t_{a,\alpha(a,k)}$ and observe that $t$ commutes with $f$ as $\supp(t)=\{a,\alpha(a,k)\}$ is disjoint with $\supp(f)$. But this contradicts the choice of $f\in O_g$.
\smallskip

2. To show that the subspace $S_{=n}(X)$ of the $T_1$-space $(G,\Zetap)$ is discrete, fix any permutation $g\in S_{=n}(X)$ with finite support $A=\supp(g)$ of cardinality $|A|=n$. Lemma~\ref{l5} implies that the set $$U=\{f\in S_{=n}(X):A\subset \supp(f)\}=\{f\in S_{=n}(A):\supp(g)= \supp(f)\}\subset G(A)$$ is relatively $\Zetap$-open in $S_{=n}(X)$ and is finite, being a subset of the finite subgroup $G(A)$. So, $g$ has finite neighborhood  $U\subset G(A)$ in $S_{=n}(X)$ and hence $g$ is an isolated point of the space $S_{=n}(X)$, which means that the space $S_{=n}(X)$ is discrete.
\end{proof}

Let us remind that a topological space $(X,\tau)$ is called {\em $\sigma$-discrete} if $X$ can be written as a countable union $X=\bigcup_{n\in\w}X_n$ of discrete subspaces of $X$. According to this definition, each countable topological space is $\sigma$-discrete. Since $\FSym(X)=\bigcup_{n\in\w}S_{=n}(X)$, Theorem~\ref{t4}(2) and Proposition~\ref{p1}(1) imply:

\begin{corollary}\label{c2} For any set $X$, the group $G=\FSym(X)$ is $\sigma$-discrete in the restricted Zariski topology $\Zetap$, and consequently $G$ is $\sigma$-discrete in each Hausdorff shift-invariant topology on $G$.
\end{corollary}

\begin{remark} Corollary~\ref{c2}  answers a problem posed in \cite{GGRC}. In \cite{BG} this corollary is generalized to so-called perfectly supportable semigroups.
\end{remark}

\section{The topologies $\Tau_\alpha$ and $\Tau_\beta$ on the symmetric group $\Sym(X)$}

It is well known that each infinite discrete topological space $X$ has two natural compactifications: the Aleksandrov (one-point) compactification $\alpha X=X\cup\{\infty\}$ and the Stone-\v Cech compactification $\beta X$. The compactifications $\alpha X$ and $\beta X$ are the smallest and the largest compactifications of $X$, respectively (see \cite[\S3.5]{En}).

Each permutation $f:X\to X$ uniquely extends to homeomorphisms $f_\alpha:\alpha  X\to\alpha X$ and $f_\beta:\beta X\to\beta X$. Conversely, each homeomorphism $f$ of the compactification $\alpha X$ or $\beta X$ determines a permutation $f|X$ of the set $X$. So, the symmetric group $S(X)$ of $X$ is algebraically isomorphic to the homeomorphism groups $\Homeo(\alpha X)$ and $\Homeo(\beta X)$ of the compactifications $\alpha X$ and $\beta X$.

It is well known that for each compact Hausdorff space $K$ its homeomorphism group $\Homeo(K)$ endowed with the compact-open topology is a topological group. If the compact space $K$ is zero-dimensional, then the compact-open topology on $\Homeo(K)$ is generated by the base consisting of the sets $$N(f,\U)=\bigcap_{U\in\U}\{g\in \Homeo(K):g(U)=f(U)\}$$ where $f\in\Homeo(K)$ and $\U$ runs over finite disjoint open covers of $K$.

The identification of $S(X)$ with the homeomorphism groups $\Homeo(\alpha X)$ and $\Homeo(\beta X)$ yields  two Hausdorff group topologies on $S(X)$ denoted by $\Tau_\alpha$ and $\Tau_\beta$, respectively.

Taking into account that each disjoint open cover of the Aleksandrov compactification $\alpha X$ can be refined by a  cover $\big\{\alpha X\setminus F\big\}\cup\big\{\{x\}:x\in F\big\}$ for some finite subset $F\subset X$, we see that the topology $\Tau_\alpha$ on $S(X)=\Homeo(\alpha X)$ coincides with the topology of pointwise convergence $\Tau_p$. The topology $\Tau_\beta$ on the symmetric group $S(X)=\Homeo(\beta X)$ is strictly stronger than $\Tau_p=\Tau_\alpha$. Its neighborhood base at the neutral element $\mathbf 1$ of $\Sym(X)$ consists of the sets $N(\mathbf 1,\U)=\bigcap_{U\in\U}\{f\in \Sym(X):f(U)=U\}$, where $\U$ runs over all finite disjoint covers of $X$.

\begin{theorem}\label{quot} The normal subgroup $\FSym(X)$ is closed and nowhere dense in the topological group $(S(X),\Tau_\beta)$.
\end{theorem}

\begin{proof} Given any permutation $f\in\Sym(X)\setminus\FSym(X)$ with infinite support $\supp(f)$, we can find an infinite subset $U\subset X$ such that $f(U)\cap U=\emptyset$. This set determines the cover $\U=\{U,X\setminus U\}$ and the $\Tau_\beta$-open neighborhood $N(f,\U)=\{g\in \Sym(X):g(U)=f(U)\}$ of $f$ which is disjoint with the subgroup $\FSym(X)$ because each permutation $g\in N(f,\U)$ has infinite support $\supp(g)\supset U$. So, the subgroup $\FSym(X)$ is closed in $(S(X),\Tau_\beta)$.

To see that $\FSym(X)$ is nowhere dense in $(S(X),\Tau_\beta)$, choose any finite cover $\U$ of $X$ and consider the basic neighborhood $N(\mathbf 1,\U)=\bigcap_{U\in\U}\{f\in S(X):f(U)=U\}\in\Tau_\beta$ of the identity element $\mathbf 1$. Since $X$ is infinite, some set $U\in\U$ is also infinite. Then we can choose a permutation $f:X\to X$ with infinite support $\supp(f)=U$. This permutation belongs to the neighborhood $N(\mathbf 1,\U)$ and witnesses that the closed subgroup $\FSym(X)$ is nowhere dense in the topological group $(S(X),\Tau_\beta)$.
\end{proof}

\begin{remark} Theorem~\ref{quot} implies that the quotient group $\Sym(X)/\FSym(X)$ admits a non-discrete Hausdorff group topology. This answers negatively Question 5.27 posed in \cite{BD}. Observe that the quotient group $\Sym(X)/\FSym(X)$ of the topological group $(\Sym(X),\Tau_\beta)=\Homeo(\beta X)$ can be naturally embedded in the homeomorphism group $\Homeo(\beta X\setminus X)$ of the remainder $\beta X\setminus X$ of the Stone-\v Cech compactification of $X$. The question if the groups $\Sym(\IZ)/\FSym(\IZ)$ and $\Homeo(\beta \IZ\setminus \IZ)$ are equal is non-trivial and cannot be resolved in ZFC; see \cite{Vel}.
\end{remark}

\begin{remark} Since for an infinite set $X$ the topology $\Tau_\beta$ on the symmetric group $\Sym(X)$ is strictly stronger than the topology $\Tau_\alpha=\Tau_p$, the homeomorphism group $\Homeo(\beta X)=(S(X),\Tau_\beta)$ is not minimal, in contrast to the homeomorphism group $\Homeo(\alpha X)=(\Sym(X),\Tau_p)$ which is minimal according to \cite{Gau}. This resembles the situation with the homeomorphism groups $\Homeo(\II^n)$ and $\Homeo(\mu^n)$ of the cubes $\II^n=[0,1]^n$ and the Menger cubes $\mu^n$ of finite dimension $n$. By a result of Gamarnik \cite{Gamarnik}, the homeomorphism group $\Homeo(\II^{n})$ is minimal if and only if $n\le 1$. By a recent result of van Mill \cite{vMill}, the homeomorphism group  $\Homeo(\mu^n)$ is minimal if and only if $n=0$. Let us observe that the zero-dimensional Menger cube $\mu^0$ is the standard ternary Cantor set, homeomorphic to the Cantor cube $\{0,1\}^\w$. The minimality of the homeomorphism group $\Homeo(\{0,1\}^\w)$ was proved by Gamarnik \cite{Gamarnik}. The problem of minimality of the homeomorphism group $\Homeo([0,1]^\w)$ of the Hilbert cube $[0,1]^\w$ is still open (cf. \cite[VI.8]{Ar87}, \cite[3.3.3(b)]{CHR}, \cite{Us08}).
\end{remark}

\section{Some Open Problems}

It is known \cite{Kal}, \cite{KR}, \cite{RS} that for a countable set $X$ the topology of pointwise convergence $\Tau_p=\Tau_\alpha$ on the permutation group $\Sym(X)$ is a unique $\w$-bounded Hausdorff group topology on $\Sym(X)$. Let us recall \cite{Gur} that a topological group $G$ is {\em $\w$-bounded} if for each non-empty open set $U\subset G$ there is a countable subset $F\subset G$ with $F\cdot U=G$. So, the topology $\Tau_p=\Tau_\alpha$ is simultaneously minimal and maximal in the class of $\w$-bounded groups topologies on $\Sym(\IZ)$.

\begin{problem} Has the topology $\Tau_\beta$ on the symmetric group $\Sym(\IZ)$ some extremal properties? In particular, is the quotient topology on the quotient group $\Sym(\IZ)/\FSym(\IZ)$ of the topological group $(\Sym(\IZ),\Tau_\beta)$ minimal? Is it a unique non-discrete Hausdorff group topology on $\Sym(\IZ)/\FSym(\IZ)$?
\end{problem}

Theorem~\ref{cent} motivates the following problem.%By Theorem~\ref{cent}, the class of groups $G$ with $\Zetapp=\Zeta=\M=\C$ contains all permutation groups $\FSym(X)$.

\begin{problem} Find a characterization of groups $G$ such that $\Zetapp=\Zeta=\M=\C$.
\end{problem}

By Proposition~\ref{p3.6}, for each infinite group $G$ with finite double centralizers the centralizer topology $\C$ is not discrete. By Proposition~\ref{p3.5} and Theorem~\ref{cent}, for each infinite set $X$ the permutation group $G=\FSym(X)$ has finite double centralizers and satisfies the equality $\M=\C$. This motivates the following problem.

\begin{problem} Is $\M=\C$ for each (countable) group $G$ with trivial center and finite double centralizers?
\end{problem}

By the classical result of Markov \cite{Mar}, a countable group $G$ admits a non-discrete Hausdorff group topology if and only if its Zariski topology $\Zeta$ is not discrete. So, for the non-topologizable groups constructed in \cite{Ol80}, \cite{KT} the Zariski topology $\Zeta$ is discrete. On the other hand, by \cite{Zel}, \cite{Zel07}, each infinite group $G$ admits a non-discrete Hausdorff topology turning it into a quasi-topological group, which implies that the restricted Zariski topology $\Zetap$ never is discrete.

\begin{problem} Is the restricted Zariski topology $\Zetapp$ discrete for some infinite group $G$?
\end{problem}

An affirmative answer to this problem implies a negative answer to the following related problem.

\begin{problem} Does each (countable) group $G$ admit a non-discrete Hausdorff  topology turning $G$ into a \textup{[}quasi\textup{]}-topological group? A \textup{[}semi\textup{]}-topological group?
\end{problem}

\section{Acknowledgement}

The authors express their sincere thanks to Dmitri Shakhmatov, Dikran Dikranjan and the anonymous referee for careful reading of the manuscript and numerous remarks and suggestions, which substantially changed the content of the paper, improved its readability and enriched the list of references.
%\newpage

\end{document}